\documentclass[final]{svjour3}
\smartqed
\usepackage{graphicx}
\usepackage{showkeys}
\usepackage{mathptmx}
\usepackage{latexsym}
\usepackage{amsmath}
\usepackage{amsfonts, textcomp, footmisc}
\usepackage{amssymb}
\usepackage{color}

\spdefaulttheorem{criterion}{Criterion}{\bf}{\it}

\newcommand{\rad}{\rm rad}
\newcommand{\disp}{\displaystyle}
\DeclareMathOperator{\dist}{dist}

\begin{document}

\title{Comparison principle, stochastic completeness and half-space theorems\thanks{G. P. Bessa partially supported by PRONEX/FUNCAP/CNPq\\
 J. H. Lira partially supported by PRONEX/FUNCAP/CNPq\\ A. Medeiros partially supported by CNPq}}

\author{G. P. Bessa  \and  J. H. de Lira \and A. A. Medeiros }

\institute{G. P. Bessa \and  J. H. de Lira \and A. A. Medeiros \at Universidade Federal do Cear\'{a} - UFC
              Departamento de Matem\'{a}tica - Campus do Pici \\
              60440-900, Fortaleza - CE, Brazil\\
              \email{bessa@mat.ufc.br, jorge.lira@mat.ufc.br, adriano.mat@hotmail.com}           
}

\date{Received: date / Accepted: date}

\maketitle

\begin{abstract}
We present a criterion for the stochastic completeness of a submanifold in terms of its distance to a hypersurface in the ambient space. This relies in a suitable version of the Hessian comparison theorem. In the sequel we apply a comparison principle with geometric barriers for establishing mean curvature estimates for stochastically complete submanifolds in Riemannian products, Riemannian submersions and wedges.  These estimates are applied for obtaining both horizontal and vertical half-space theorems for submanifolds in $\mathbb{H}^n \times \mathbb{R}^\ell$.
\end{abstract}
\tableofcontents
\section{Introduction}\label{intro}

In  recent years, an intense research effort has been redefining the geometric analysis on manifolds and submanifolds in terms of notions and theorems coming from the theory of stochastic processes. A fundamental piece for  bridging  these two subjects is the beautiful A. Grigor'yan's survey \cite{grigoryan} where he approached important stochastic notions via potential theory, what permitted apply  stochastic tools  to geometric elliptic equations in  deterministic terms.

The second landmark in this \textcolor{black}{way} is the  work by S. Pigola, M. Rigoli and A. Setti  in \cite{PRS-PAMS} where they proved that  stochastic completeness of  Riemannian manifolds is  equivalent to the validity of the weak maximum principle at infinity. Having established this link, Pigola, Rigoli and Setti set up  the tools for an impressive wealth of applications of the maximum principle in the analysis of geometric PDEs \cite{bessa-pigola-setti-revista},  \cite{PRS-gafa}, \cite{prs-memoirs},  \cite{prs-vanishing}, \cite{PSZ-JMPA}, \cite{puriser-jde},  \cite{yau2}. These tools are particularly useful in obtaining mean  curvature estimates on submanifolds \textcolor{black}{subject} to some kind of extrinsic bound, see \cite{albanese-alias-rigoli}, \cite{alias-bessa-dajczer-MathAnn}, \cite{alias-bessa-montenegro}, \cite{atsuji}, \cite{bessa-lima-pessoa} \cite{jorge-xavier-bms}, \cite{jorge-xavier-mathZ}, \cite{omori}, the main topic in this article.

\textcolor{black}{We are also concerned in this paper with the} comparison theory based on the distance to a hypersurface and the \textcolor{black}{companion} use of the Riccati equation. This is the subject of a series of papers by J.-H. Eschenburg \cite{eschenburg-manuscripta}, \cite{eschenburg-heintze-manuscripta} based on an idea of M. Gromov \cite{gromov}.  Nowadays this approach to Comparison Geometry is presented even in textbooks as an alternative  to the standard one using Jacobi fields.  In spite of this we needed to develop here a function theoretic version of the Hessian comparison theorem for the distance to a hypersurface which fits well with the generalized weak maximum principle by S. Pigola, M. Rigoli and A. Setti. This allowed us to enlarge the  range of application of this technique as well as to update in this sense the work by Eschenburg to the new trends on geometric analysis of submanifolds.
More precisely, we replace conditions on the ambient curvature with requirements on the existence of functions satisfying very mild analytical restrictions.
This slight change of point of view  \textcolor{black}{opens} new perspectives to applications of the stochastic tools to the theory of minimal and constant mean curvature submanifolds.

A Riemannian manifold $M$ is said to be stochastically complete if  the Wiener process associated
to the Laplacian $\Delta$ is stochastically complete.  This means that  the probability  that a particle in  Brownian motion in $M$ starting at $x\in M$  be found in $M$ at a time $t>0$ is $1$.
This is,  the heat kernel  $p$ of $\Delta$ satisfies the conservation property:  for all $\, x\in M$ and $t>0$ fixed,
\begin{equation}
\label{heatkernel}
\int_{M}p\left(x,y,t\right)  dy
=1.
\end{equation}

On the other hand,  the weak maximum principle (at infinity) holds on a Riemannian manifold
$M$ if for every $u\in C^{2}\left(M\right)  $, with $u^{\ast
}\colon=\sup_{M}u<+\infty$, there exists a sequence $\left\{  x_{k}\right\}\subset M  $ along which
\begin{eqnarray}
\label{wmp}
\begin{array}{lll}
\text{(i) }\,\,u\left(  x_{k}\right)  >u^{\ast}-\disp\frac{1}{k},&&\text{ (ii) }\,\,
\Delta u\left(  x_{k}\right)  <\disp\frac{1}{k}\cdot\end{array}
\end{eqnarray}

The  mean curvature estimates of submanifolds subject to \textcolor{black}{the extrinsic bounds} mentioned above are obtained using the weak maximum principle in the following way:
\textcolor{black}{consider an isometric immersion} $\varphi \colon M \to N$  of a complete Riemannian manifold $M$ into a complete Riemannian manifold $N$. Set  $u=g\circ\rho_{N}\circ \varphi$, where  $\rho_{N}={\rm dist}_N(o,\,\cdot\,)$ is the ambient distance function to a given reference point $o\in N$ and $g\in C^\infty(N)$ is a function defined by some geometric rationale.
The extrinsic bounds \textcolor{black}{to that $\varphi$ is constrained imply}   that  $u^{\ast}<\infty$.  \textcolor{black}{Suppose} in addition
that $M$ is stochastically complete and that there exist upper bounds for the sectional curvatures of the ambient manifold $N$. \textcolor{black}{Under these assumptions} we can apply the  Hessian comparison theorem \cite{GW}, \cite{petersen}, \cite[p.11]{prs-memoirs} or \cite[Thm. 2.3]{prs-vanishing} to $\varrho_N$.  \textcolor{black}{Hence}  we   obtain an explicit  lower bound for $\Delta u(x_k)$ in terms of the mean curvature ${\bf H}$ of $\varphi$ what, jointly with (\ref{wmp}-\emph{ii}), yields the desired estimate to $|{\bf H}|$.

That is, roughly stated, the technique applied in \cite{alias-bessa-dajczer-MathAnn} for proving  that,  given a complete Riemannian manifold $P^{n-\ell}$ and a proper isometric immersion $\varphi:M^m\to P^{n-\ell}\times\mathbb{R}^\ell$, with $m\ge  \ell +1$ and
\[
\varphi(M) \subset B_q^{P}(r) \times\mathbb{R}^{\ell},
\]
where $B_q^{P}(r)$ is a geodesic ball in $P$ centered at $q\in P$ with $r<\textrm{inj}_P (q)$, then
\[
|{\bf H}|\ge \frac{m-\ell}{m} \frac{1}{\sqrt{-\kappa}}\coth (\sqrt{-\kappa}r),
\]
provided that the radial sectional curvatures of $P$ along geodesics issuing from a given  point $q'\in \pi(\varphi(M))$ satisfy
\[
K_N^{\textrm{rad}}\le \kappa,
\]
for some constant $\kappa<0$. Here $\pi:P\times \mathbb{R}^{\ell}\to P$ is the standard projection. This is indeed a very particular consequence of the Theorem 6 in \cite{alias-bessa-dajczer-MathAnn}.

In this paper, we will consider applications of the weak maximum principle using the distance function $\varrho$ to an embedded oriented hypersurface $\Sigma_0\subset N$  instead of the distance function $\rho_N$  to a reference point $o\in N$. The corresponding Hessian comparison theorem (Theorem \ref{hessian} below) allows us to obtain estimates for the mean curvature in a even wider class of ambient spaces as Riemannian submersions.  We obtain a  comparison principle that is used to prove, in different settings, \textcolor{black}{a number of} mean curvature estimates. This principle is detailed in the following result, Theorem \ref{thm-product-intro}, a particular case of Theorem \ref{submersion}  in Section \ref{sec:comparison}. It gives mean curvature estimates  for stochastically complete submanifolds immersed in regular tubes around embedded hypersurfaces of product spaces. This naturally extends  the main result of \cite{alias-bessa-dajczer-MathAnn}.

 Let $N$ be  a complete Riemannian $n$-manifold  and assume that there exists an oriented,  complete, embedded  hypersurface $\Sigma_0\subset N$ and let $\varrho = \dist(\cdot,  \Sigma_0)$ be the \emph{signed} distance function from $\rho^{-1}(0)=\Sigma_0$ which is supposed to be regular in some region $\mathcal{U}\subset N$ free of focal points of $\Sigma_0$.
 Suppose that the radial sectional curvatures $K_N^{\rad}|_p, \,p\in \mathcal{U}$, along $\rho$-minimizing geodesics issuing orthogonally from $\Sigma_{0}$ are bounded above by some function $G\in C^\infty(\mathbb{R})$. More precisely, \begin{equation}
\label{Krad2}
K_N^{\rad}({\sf v}\wedge \bar\nabla \varrho(p))\le -G(\varrho(p)),
\end{equation} for every $p\in \mathcal{U}$ and every unit length vector ${\sf v}\in T_p N$ with $\langle {\sf v},\bar{\nabla}\varrho(p)\rangle=0$, where $\bar{\nabla}$ is the Riemannian connection in $N$.
Associate to $G$   the solution $h$ of the ODE
\begin{align}\label{eqh}
\begin{array}{l}
   h''(t)-G(t)h(t)=0,
  \end{array}
\end{align} and let $(d_*,d^*)$ be the maximal interval where $h$ is positive with $d_*<0<d^*$.

\begin{theorem}\label{thm-product-intro} Let $\varphi\colon M \to N\times L$ be an  isometric immersion of a  stochastically complete Riemannian $m$-manifold $M$ into  $N\times L$, where $L$ is a complete Riemannian $\ell$-manifold. Define
$\widetilde\varrho\colon \mathcal{U}\times L \to \mathbb{R} $, the lifting of $\varrho$
given by
$\widetilde\varrho (x,y)=\varrho (x)$. Suppose that
\begin{itemize}\item[1.] $\bar{\nabla}^2\varrho\big|_{T\Sigma_0}  \ge\disp \frac{h'(0)}{h(0)} \langle\,\,,\,\,\rangle\big|_{T\Sigma_0}$.
\item[]\item[2.]  $\varphi(M)\subset \mathcal{U}\cap \widetilde \varrho^{-1}\left((-\infty, d]\right)$ for some $d>0$ and  $\varrho(\varphi(x_0))\ge 0$ for some $x_0\in M$.
\item[]\end{itemize}
Then
\begin{equation}
\sup_{M}|{\bf H}| \geq \displaystyle\frac{m-\ell}{m}\inf_{ t\in [0, d]}\frac{h'(t)}{h(t)}\cdot
\end{equation}
In particular, there are no stochastically complete submanifolds in  $\mathcal{U}\cap \widetilde \varrho^{-1}\left((-\infty, d]\right)$  with mean curvature vector \textcolor{black}{satisfying} $|{\bf H}|< \displaystyle\frac{m-\ell}{m}\inf_{t\in [0, d]}\frac{h'(t)}{h(t)}\cdot$
\end{theorem}

This   paper is organized as follows. In section \ref{sec:comparison}, we establish a  Hessian comparison theorem for a signed distance to an oriented embedded  hypersurface $\Sigma_{0}\subset N$. In Section \ref{sec:criteria} we present a version of  Pigola-Rigoli-Setti's criterion for stochastically completeness in the setting of a distance to a hypersurface. In section \ref{sec:generalized-comparison} we prove a comparison principle which is used in the sequel for establishing several mean curvature estimates. These estimates lead to the proof of Theorem \ref{thm-product-intro} as well as to some interesting geometric results.  Indeed,  we show in Section \ref{half-space-section} how the mean curvature estimates are used for obtaining both vertical and horizontal half-space theorems for constant mean curvature submanifolds in product spaces $\mathbb{H}^{n}\times \mathbb{R}^\ell$. Our half-spaces theorems generalize, with simpler proofs, some recently proven results for surfaces in $\mathbb{H}^{2}\times \mathbb{R}$. Finally in Section \ref{half-space-section} we establish mean curvature estimates for stochastically complete submanifolds immersed in a wedge in a Riemannian product. These estimates improve results of \cite{bessa-lima-pessoa} and \cite{mari-rigoli}.

\section{Hessian comparison theorem}\label{sec:comparison}

Let $N$ be a Riemannian manifold and $\Sigma_0\subset N$  a complete embedded oriented hypersurface. Let $\varrho\colon N \to \mathbb{R}$,
$\varrho(x)=\dist (x,\, \Sigma_0)$,  be the \emph{signed} distance function from $\rho^{-1}(0)=\Sigma_0$ which is supposed to be regular in some region $\mathcal{U}\subset N$ free of focal points of $\Sigma_0$.  The following theorem is an extension of the Hessian comparison theorem, see \cite{GW}, \cite{petersen}, \cite[p.11]{prs-memoirs} and \cite[Theorem 2.3]{prs-vanishing}.

\begin{theorem}[Hessian Comparison Theorem] \label{hessian}  Suppose that there exists an even smooth function $G\colon\mathbb{R}\to\mathbb{R}$ such that
\begin{equation}
\label{comp-K2-a}
K_N({\sf v}\wedge \bar{\nabla}\varrho(p))\ge -G(\varrho(p)),
\end{equation}
for every $p$ in some open domain $\mathcal{U}\subset N$ where $\varrho$ is regular and every unit length vector ${\sf v}\in T_p N$ with $\langle {\sf v},\bar{\nabla}\varrho(p)\rangle=0$. Let $h$ be a solution of the ODE
\begin{equation}
\label{jacobi2-a}
h''-Gh=0.
\end{equation}
Let $(d_*,d^*)$ the maximal interval where $h$ is positive with $d_*<0<d^*$. If
\begin{equation}
\label{comp-A2-a}
\bar{\nabla}^2\varrho\big|_{T\Sigma_0}  \le \frac{h'(0)}{h(0)} \langle\,\, ,\,\rangle\big|_{T\Sigma_0}
\end{equation}
then
\begin{equation}
\label{comp-AA2}
\bar{\nabla}^2 \varrho({\sf v}, {\sf v}) \le \frac{h'}{h}\langle {\sf v}, {\sf v}\rangle,
\end{equation}
for any $(p, {\sf v})\in T\mathcal{U}$ with $\langle {\sf v}, \bar{\nabla}\varrho(p)\rangle=0$, whenever $0\le \varrho(p)<d^*$.
\vspace{2mm}

\noindent If instead of \eqref{comp-K2-a} and  \eqref{comp-A2-a} we have
\begin{eqnarray}
\label{comp-reverse1}
K_N({\sf v}\wedge \bar{\nabla}\varrho(p))& \leq & -G(\varrho(p)) \end{eqnarray} and \begin{eqnarray}\bar{\nabla}^2\varrho\big|_{T\Sigma_0}& \geq & \disp \frac{h'(0)}{h(0)} \langle\,\,,\,\,\rangle\big|_{T\Sigma_0}\label{comp-reverse2}
\end{eqnarray}
 then
\begin{equation}
\label{conc-reverse}
\bar{\nabla}^2 \varrho({\sf v}, {\sf v}) \ge \frac{h'}{h}\langle {\sf v}, {\sf v}\rangle.
\end{equation}
Here $\bar{\nabla}\varrho$ and $\bar{\nabla}^2\varrho$ are  respectively the gradient and the hessian of $\varrho$ and  $K_N({\sf v}\wedge \bar{\nabla}\varrho)$ is the sectional curvature of the two-plane ${\sf v}\wedge \bar{\nabla}\varrho$ along the geodesics issuing orthogonally from $\Sigma_0$.

\end{theorem}

\begin{proof}\hspace{-1mm}: The proof of this result is based on the proofs of \cite[Thm. 3.1]{eschenburg-manuscripta} and \cite[Thm. 2.3]{prs-vanishing}.
Using Fermi coordinates  in $\mathcal{U}$ it is  easily verified  the validity of the well-known Riccati's equation (see \cite{eschenburg-manuscripta} for details)
\begin{equation}
\label{riccati-a}
\bar{\nabla}_{\bar{\nabla}\varrho}A + A^2 + \mathcal{R} =0,
\end{equation}
where $A(X)=\bar{\nabla}_{_X}\bar{\nabla}\varrho$ and the endomorphism $\mathcal{R}$ is given by
\[
\mathcal{R} ({\sf v}) = \bar R({\sf v}, \bar{\nabla}\varrho)\bar{\nabla}\varrho
\]
and $\bar{R}$ denotes the Riemann tensor in $N$.   Fix a minimizing geodesic $\gamma:[-d_0,d_0]\to N$ so that it crosses $\Sigma_0$ orthogonally at $d=0$.  Assume  that
\begin{equation}
\label{contradiction-a}
K_N({\sf v}\wedge \gamma'(d))> -G(\varrho(d))
\end{equation}
for $d\in (-d_0,d_0)$. Let $A(t)=A|_{\gamma(t)}$ and $\widetilde{A}(t)=\disp\frac{h'}{h}(t)Id|_{\gamma(t)}$ be
the restrictions of the $(1,1)$-tensors $\bar{\nabla}_{_{(\cdot)}}\bar{\nabla}\varrho$ and $\disp\frac{h'}{h}Id$
to $\gamma'^\perp$ along $\gamma$. Suppose that $h$ has no zeroes in $[0,d_0]$, this is,  $d_0<d^*$.
Recall that $A(t)\leq \widetilde{A}(t)$ if and only if $\langle A(t){\sf v},{\sf v}\rangle \leq \langle \widetilde{A}(t){\sf v},{\sf v}\rangle$ for all ${\sf v}\in \gamma'^{\perp}\subset T_{\gamma(t)}N$. Define $$\delta=\sup\{d\in (0,d_0]\colon A(d)\le \widetilde A(d)\}.$$ Suppose that $\delta<d_0$. By continuity, $A(\delta)\le \widetilde A(\delta)$. We claim that the strict inequality holds indeed. Otherwise, $\widetilde A(\delta)-A(\delta)\ge 0$ has  nontrivial kernel. Hence if ${\sf v}\in T_{\gamma(\delta)} N$ is a null eigenvector of this operator, then the function
\[
\lambda(t) = \langle (\widetilde  A-A){\sf v}(t),{\sf v}(t)\rangle
\]
is nonnegative in $[0,\delta]$ with a zero at $\delta$, where ${\sf v}(t)$ is the parallel transport of ${\sf v}$ along $\gamma|_{[0,\delta]}$. Denoting $\phi= \displaystyle\frac{h'}{h}$ and using (\ref{jacobi2-a}) and (\ref{riccati-a}) we compute
\begin{eqnarray*}
\lambda'(d)  &=&  \langle (A^2-\phi^2 \cdot Id){\sf v}(d),{\sf v}(d)\rangle + \langle (\mathcal{R}+G\cdot Id){\sf v}(d),{\sf v}(d)\rangle\\
& = & \langle (A+\phi \cdot Id)(A-\phi \cdot Id){\sf v},{\sf v}\rangle + \langle (\mathcal{R}+G\cdot Id){\sf v},{\sf v}\rangle.
\end{eqnarray*}By the fact that $A{\sf v}=\widetilde A{\sf v}$ at $\gamma(\delta)$ one has
\[
\lambda'(\delta) =\langle (\mathcal{R}+G \cdot Id){\sf v},{\sf v}\rangle|_{\gamma(\delta)}>0
\]
in view of (\ref{contradiction-a}). However, this contradicts the fact that $\lambda|_{[0,\delta]}\ge 0$ and $\lambda(\delta)=0$. This proves the claim. By continuity we have $A\le \widetilde A$ in an interval of the form $[\delta, \delta+\varepsilon)$, for some $\varepsilon>0$. This contradicts the definition of $\delta$. This contradiction implies  that $\delta=d_0$. More precisely, $A\le \widetilde A$ in the whole interval $[0,d_0]$. In particular, if $\widetilde d^*>0$ is the first instant at which $\widetilde A$ is singular, then  $d_0\le \widetilde d^*$. Finally the strict inequality in (\ref{contradiction-a}) may be relaxed by continuity. The case with inverted inequality is similar. This finishes the proof. \hfill $\square$
\end{proof}

\begin{remark}\label{geometries}
In geometric terms, the conditions (\ref{comp-K2-a})-(\ref{comp-A2-a}) compare Riccati's equations in $N$ and in a model  space $\widetilde N= (d_*, d^*)\times_h \mathbb{S}^{n-1}$. Following \cite[Theorem 3.1]{eschenburg-manuscripta} we compare eigenvalues of the Weingarten map of $\Sigma_0$ and of the leaf $\widetilde\Sigma_0=\{0\}\times \mathbb{S}^{n-1}$. This leaf corresponds to the zero level set of the distance function $\widetilde\varrho=\dist(\cdot, \widetilde \Sigma_0)$ in $\widetilde N$. In this way, (\ref{comp-A2-a}) implies that
\[
\min_{{\sf v}} \langle A{\sf v}, {\sf v}\rangle  \ge \max_{\widetilde{\sf v}}\langle \widetilde A {\sf v}, {\sf v}\rangle,
\]
where the minimum and maximum are taken, respectively,  over ${\sf v}\in T\Sigma_0$ and $\widetilde{\sf v}\in T\widetilde\Sigma_0$ with $|{\sf v}|=|\widetilde{\sf v}|=1$,
what parallels exactly the main assumption in \cite[Theorem 3.1]{eschenburg-manuscripta}.
\end{remark}

\section{Pigola-Rigoli-Setti's criterion for stochastic completeness}\label{sec:criteria}
An important property of the class of  stochastically complete submanifolds is the validity of weak maximum principle at infinity. Thus. it would be useful and necessary to know when a given submanifold is stochastically complete in terms of their intrinsic and extrinsic geometries.  For instance, D. Stroock \cite{stroock} proved that properly embedded minmal surfaces of $\mathbb{R}^{3}$  are stochastically complete. A. Kasue \cite{kasue} extended  Stroock's theorem to  complete immersed submanifolds of $\mathbb{R}^{n}$ with bounded mean curvature. These results were all extended by Pigola, Rigoli and Setti, \cite[Thm. 1.9 \& Example 1.14]{prs-memoirs}, in a criterion that gives general geometric conditions for an  immersed submanifold to be stochastically complete. Indeed they proved the following
\begin{criterion}[Pigola-Rigoli-Setti]\label{criterion1}
Let $\varphi\colon M \to N$ be a proper isometric immersion of a  complete Riemannian manifold $M$ into a complete Riemannian manifold $N$. Let $p\in N\backslash \varphi(M)$ and suppose that either  ${\rm cut}_{N}(p)=\emptyset$ or $\varphi (M)\cap {\rm cut}_{N}(p)=\emptyset$. Let $\varrho_N = \dist_N(\cdot, p)$ be the distance to the point $p$. If the radial sectional curvatures of $N$ along the geodesics issuing from $p$ satisfies
\begin{equation}
\label{Krad}
K_{N}^{\rad}(q)\geq - G(\rho_{N}(q)), \quad q \in N\backslash ({\rm cut}_N(p)\cup \{p\}),
\end{equation}
and the mean curvature vector ${\bf H}$ of the immersion $\varphi$ satisfies
\begin{equation}
\label{HG}
\vert {\bf H} \vert (\varphi(x))\leq B\,\varphi\sqrt{G(\rho_{N}\circ \varphi(x))}, \quad x\in M,
\end{equation}
for some positive constant $B$ and a smooth function $G$ defined on $[0, \infty)$ satisfying
\begin{equation}
\label{G}
G(0)>0,\,\,\,\,G'(t)\geq 0, \,\,\,\,G^{-1/2}(t)\not\in L^{1}(+\infty), \,\,\,\,\limsup_{t\to\infty}\frac{tG(t^{1/2})}{G(t)}<+\infty
\end{equation}
on $[0,\infty)$, then $M$ is stochastically complete.\label{criterion}
\end{criterion}
\begin{remark}The criterion \ref{criterion1} holds when $M=N$. In this case, the mean curvature ${\bf H}$ of $M$  is zero, thus  the Criterion \ref{criterion1} says that if the radial sectional curvature $K_{M}^{\rad}$ decays as in \eqref{Krad} then $M$ is stochastically complete.
\end{remark}
\begin{remark}Example of functions satisfying \eqref{G} are given by \[G(t)=t^2\cdot\disp\Pi_{i=1}^{k}\left(\log^{(i)}(t)\right)^{2}, \,\,t\gg 1\]where $\log^{(i)}$ stands for the $i$-th iterated logarithm.\end{remark}
This criterion may be  extended, replacing $\varrho_N$ by the distance from an embedded oriented hypersurface  $\Sigma_0\subset N$ with  the same reasoning as in \cite[Examples 1.13, 1. 14]{prs-memoirs}.
\begin{theorem}
Let $\varphi\colon M \to N$ be an  isometric immersion of a  complete Riemannian manifold $M$ into a complete Riemannian manifold $N$. Suppose there exists an embedded oriented complete hypersurface  $\Sigma_0\subset N$ so that the signed distance function $\varrho = \dist(\cdot, \Sigma_0)$ is regular function in some domain  $\mathcal{U}\subset N$, that   $\varphi(M)\subset \mathcal{U}$ and that $\varrho \circ \varphi $ is proper.  If the radial sectional curvatures of $N$ along the geodesics issuing perpendicularly from $\Sigma_0$ satisfies
\begin{equation}
\label{Krad2}
K_{N}^{\rad}(p)\geq -G(\varrho(p)), \quad p \in \mathcal{U},
\end{equation}
and the mean curvature vector ${\bf H}$ of the immersion $\varphi$ satisfies
\begin{equation}
\label{HG2}
\vert {\bf H} \vert (\varphi(x))\leq B\, \sqrt{G(\varrho\circ \varphi(x))}, \quad x\in M,
\end{equation}
for some positive constant $B$ and a smooth function $G$ defined on $[0, \infty)$ satisfying
\begin{equation}
\label{G2}
G(0)>0,\,\,\,\,G'(t)\geq 0, \,\,\,\,G^{-1/2}(t)\not\in L^{1}(+\infty)
\end{equation}
on $[0,\infty)$, then $M$ is stochastically complete.\label{criterion2}
\end{theorem}
\begin{remark}It follows from Borb\'{e}ly's work \cite{borbely-Bull-Aus} that this criterion holds without the condition $\limsup_{t\to\infty}\frac{tG(t^{1/2})}{G(t)}<+\infty$, see also \cite[Thm.9]{bessa-pigola-setti-revista}.
\end{remark}
\begin{proof}\hspace{-1mm}:
Denote by  $\nabla$ and $\bar{\nabla}$ the Riemannian covariant derivatives in $M$ and $N$, respectively.
Consider the restriction $\varrho\circ\varphi$ of $\varrho$ to the hypersurface $\varphi(M)$. Hence we have
\begin{equation}
\Delta (\varrho\circ\varphi)^2 = \sum_{i=1}^m \langle \bar{\nabla}_{\varphi_* {\sf e}_i} \bar{\nabla} \varrho^2, \varphi_*{\sf e}_i\rangle +  \langle \bar{\nabla}\varrho^2, m{\bf H}\rangle,
\end{equation}
where $\{{\sf e}_i\}_{i=1}^m$ is a local orthonormal frame tangent to $M$ and ${\bf H}$ is the mean curvature vector field along $\varphi$. Theorem \ref{hessian}
implies that
\begin{eqnarray*}
\Delta (\varrho\circ\varphi)^2 &\le & 2\varrho \frac{h'}{h}\sum_{i=1}^m |\varphi_* {\sf e}_i^\perp|^2 + 2\varrho\sum_{i=1}^m\langle \bar{\nabla}\varrho, \varphi_*{\sf e}_i\rangle^2 +  2\varrho\langle \bar{\nabla}\varrho, m{\bf H}\rangle\\
& = & 2\varrho \frac{h'}{h}\sum_{i=1}^m |\varphi_* {\sf e}_i|^2 + 2\varrho\Big(1-\frac{h'}{h}\Big)\sum_{i=1}^m\langle \bar{\nabla}\varrho, \varphi_*{\sf e}_i\rangle^2 +  2\varrho\langle \bar{\nabla}\varrho, m{\bf H}\rangle
\end{eqnarray*}
where $\perp$ denotes the tangent projection on the equidistant hypersurfaces $\Sigma_d =\varrho^{-1}(d)$ and $h:[d_*, \infty)\to \mathbb{R}$ is the solution of the initial value problem
\begin{equation}
\label{hC}
h''-G h=0, \quad h(d_*)=0, \quad h'(d_*)=1,
\end{equation}
for some $d_*<0$, which satisfies (\ref{comp-A2-a}) at $d=0$.

Define
\begin{equation}
\widetilde h(d) = \frac{1}{D\sqrt{G(d_*)}}\Big(e^{D\int_{d_*}^d \sqrt{G(\tau)}\textrm{d}\tau}-1\Big).
\end{equation}
Hence $\widetilde h (d_*)=0,\, \widetilde h'(d_*)=1$ and
\[
\widetilde h''\ge \frac{G}{\sqrt{G(d_*)}}\Big(D+\frac{1}{2}\frac{G'}{G^{3/2}}\Big)e^{D\int_{d_*}^d \sqrt{G(\tau)}\textrm{d}\tau}\ge\frac{1}{D}\frac{G}{\sqrt{G(d_*)}}e^{D\int_{d_*}^d \sqrt{G(\tau)}\textrm{d}\tau}\ge G\widetilde h,
\]
for $D>0$ sufficiently large. It follows from the Sturm comparison theorem,  \cite[Lem. 2.1]{prs-vanishing} that $\widetilde h \ge h$ in $(d_*, d^*)$ and
\begin{eqnarray*}
\frac{h'}{h} \le \frac{\widetilde h'}{\widetilde h} \le D\sqrt{G}.
\end{eqnarray*}
in this interval. Therefore, there exists a compact subset $K\subset M$ such that
\begin{eqnarray*}
\Delta (\varrho\circ\varphi)^2 \le 2mD\sqrt{G(\varrho)}\varrho +  2m\varrho |{\bf H}|\le C\sqrt{G(\varrho)}\varrho
\end{eqnarray*}
holds in $M\backslash K$, for some constant $C>0$.
We conclude that $(G,(\varrho\circ\varphi)^2)$ is a Omori-Yau pair in $M$ (see \cite[Def.3.3]{alias-bessa-montenegro-piccione}
for a precise definition\footnote{Taking in account Borb\'{e}ly's work \cite{borbely-Bull-Aus}, may drop the condition $\limsup_{t\to\infty}\frac{tG(t^{1/2})}{G(t)}<+\infty$ in the definition of  the Omori-Yau pair.}). This implies that $M$ is stochastically complete. \hfill $\square$
\end{proof}

\section{A generalized mean curvature comparison principle}
In this section we prove mean curvature estimates for immersed submanifolds. The basic setting in all of the results from now on is as follows.
\label{sec:generalized-comparison}Let $\varphi\colon  M \to N$ be a proper isometric immersion of a complete Riemannian manifold $M$ into a complete Riemannian manifold $N$. Let   $\Sigma_0\subset N$ be a complete oriented embedded hypersurface such that   the signed distance $\varrho = \dist(\cdot, \Sigma_0)$ from $\Sigma_0$  is regular in a domain $\mathcal{U}\subset N$. Suppose that the radial sectional and mean curvatures of $N$ and $\varphi(M)$ along the geodesics issuing perpendicularly from $\Sigma_0$ satisfy the inequality \eqref{Krad2}.\footnote{If, in addition, the inequality \eqref{HG2} is satisfied then $M$ is stochastically complete. }  Assume that $\Sigma_{0}$ is such that the solution  $h:[d_*, \infty)\to \mathbb{R}$, $d_*<0$,  of the initial value problem \eqref{hC}  satisfies (\ref{comp-A2-a}) in $T\Sigma_0$.
Denote $f=g\circ\varrho\circ\varphi$ where
\begin{equation}
g(\varrho) = \int_{\varrho_0}^\varrho h(r)\, \textrm{d}r,
\end{equation}
for some $\varrho_0\ge d_*$. In this setting above,
the Laplacian of $f$ can be  related to the mean curvature $\vert {\bf H}\vert$ of the submanifold $\varphi(M)$  and to the mean curvature $H_d$ of  the equidistant hypersurface $\Sigma_d=\varrho^{-1}(d)$ with respect to $-\bar{\nabla}\varrho$ in an  expression  that will be used repeatedly in the sequel.

\begin{proposition}\label{mainformula}
Let $\varphi\colon M \to N$ be a proper isometric immersion of a  complete Riemannian manifold $M$ into a complete Riemannian manifold $N$. Let $\Sigma_0$ be a complete hypersurface in $N$ and consider the signed distance function $\varrho = \dist(\cdot, \Sigma_0)$. Suppose that  $\varphi(M)$ is contained in some domain $\mathcal{U}\subset N$ where $\varrho$ is a regular function. Suppose that $\inf_\mathcal{U}h>0$. If \emph{(\ref{comp-K2-a})} and \emph{(\ref{comp-A2-a})} in Theorem \ref{hessian} hold then
\begin{equation}
\label{Deltaf}
\frac{\Delta f}{h} \ge (n-1)H_d-(n-1)\frac{h'}{h}- m|{\bf H}| + m\frac{h'}{h}.
\end{equation}
Moreover $H_d\le\disp\frac{h'(d)}{h(d)}$ in $\mathcal{U}$.
\end{proposition}

\begin{proof}\hspace{-1mm}:
The Laplacian of $f$ is calculated as follows
\begin{eqnarray*}
\Delta  f &=& g' \Delta(\varrho\circ\varphi) + g''|\nabla (\varrho\circ\varphi)|^2 \\
& = & g'\sum_{i=1}^m \langle \bar{\nabla}_{\varphi_* {\sf e}_i} \bar{\nabla} \varrho, \varphi_*{\sf e}_i\rangle + g' \langle \bar\nabla\varrho, m{\bf H}\rangle + g''|\nabla (\varrho\circ\varphi)|^2,
\end{eqnarray*}
where $\{{\sf e}_i\}_{i=1}^m$ is a local orthonormal frame tangent to $M$ and ${\bf H}$ is the mean curvature vector field along $\varphi$.  We proceed considering a local orthonormal frame $\{{\sf v}_i\}_{i=1}^{n-1}$ adapted to an equidistant $\Sigma_d =\varrho^{-1}(d)$ which intersects $\varphi(M)$.  Denoting by $\{\xi_\alpha\}_{\alpha=1}^{n-m}$ a local orthonormal frame in the normal bundle of $\varphi$ and by $\perp$ the projection on $T\Sigma_d$ we obtain
\begin{eqnarray*}
\Delta f &=& g'\Delta_N \varrho - g'\sum_{\alpha=1}^{n-m}\langle \bar{\nabla}_{\xi_\alpha} \bar{\nabla}\varrho, \xi_\alpha\rangle + g' \langle \bar{\nabla}\varrho, m{\bf H}\rangle + g''|\nabla (\varrho\circ\varphi)|^2\\
& = &
 g'\sum_{i=1}^{n-1} \langle \bar{\nabla}_{{\sf v}_i} \bar{\nabla} \varrho, {\sf v}_i\rangle - g'\sum_{\alpha=1}^l\langle \bar{\nabla}_{\xi_\alpha} \bar{\nabla}\varrho, \xi_\alpha\rangle  + g' \langle \bar{\nabla}\varrho, m{\bf H}\rangle + g''\Big(1-\sum_{\alpha=1}^{n-m} \langle\bar{\nabla} \varrho, \xi_\alpha\rangle^2\Big)\\
& = & (n-1) H_d \,g'- g'\sum_{\alpha=1}^{n-m}\langle \bar{\nabla}_{\xi_\alpha^\perp} \bar{\nabla}\varrho, \xi_\alpha^\perp\rangle  + g' \langle \bar{\nabla}\varrho, m{\bf H}\rangle + g''\Big(1-\sum_{\alpha=1}^{n-m} \langle\bar\nabla \varrho, \xi_\alpha\rangle^2\Big),
\end{eqnarray*}
where $H_d$ is the mean curvature of $\Sigma_d$ with respect to $-\bar{\nabla}\varrho$.
It follows from Theorem \ref{hessian} that
\begin{eqnarray*}
\Delta f & \ge &  (n-1) H_d \,g'- g'\frac{h'}{h}\sum_{\alpha=1}^{n-m}|{\xi_\alpha^\perp}|^2 + g' \langle \bar{\nabla}\varrho, m{\bf H}\rangle + g''\Big(1-\sum_{\alpha=1}^{n-m} \langle\bar{\nabla} \varrho, \xi_\alpha\rangle^2\Big)\\
& = & (n-1) H_d \,h- h'\sum_{\alpha=1}^{n-m}|{\xi_\alpha^\perp}|^2 + h \langle \bar{\nabla}\varrho, m{\bf H}\rangle + h'\Big(\sum_{\alpha=1}^{n-m}|{\xi_\alpha^\perp}|^2-n+m+1\Big)\\
& \ge  &  h \big((n-1)H_d - m|{\bf H}|) - (n-1-m)h'.
\end{eqnarray*}
We conclude that
\begin{equation}
\frac{\Delta f}{h} \ge (n-1)H_d-(n-1)\frac{h'}{h}- m|{\bf H}| + m\frac{h'}{h}.
\end{equation}
Notice that the Hessian comparison theorem implies that $H_d \le h'(d)/h(d)$. This finishes the proof of the proposition \hfill $\square$.
\end{proof}

\begin{remark} Some of these calculations above was also done  in \cite{bessa-pigola-setti-pota}. The inequality \eqref{Deltaf} is effective if the parallel hypersurfaces have mean curvature $H_d=h'(d)/h(d)$. Otherwise, suppose that (\ref{comp-reverse1}) and \eqref{comp-reverse2} in Theorem \ref{hessian} hold.
It follows from the proof of Proposition \ref{mainformula} that
\begin{equation}
\label{mainformula-reverse}
\frac{\Delta f}{h} \ge m \Big(\frac{h'}{h}-|{\bf H}|\Big).
\end{equation}
Indeed
\begin{eqnarray*}
\Delta f & = & g'\sum_{i=1}^m \langle \bar{\nabla}_{\varphi_* {\sf e}_i} \bar{\nabla} \varrho, \varphi_*{\sf e}_i\rangle + g' \langle \bar\nabla\varrho, m{\bf H}\rangle + g''|\nabla (\varrho\circ\varphi)|^2\\
& = & g'\sum_{i=1}^m \langle \bar{\nabla}_{(\varphi_* {\sf e}_i)^\perp} \bar{\nabla} \varrho, (\varphi_*{\sf e}_i)^\perp\rangle + g' \langle \bar\nabla\varrho, m{\bf H}\rangle + g''|\nabla (\varrho\circ\varphi)|^2.
\end{eqnarray*}
Hence Theorem \ref{hessian} implies
\begin{eqnarray*}
\Delta f & \ge & g' \frac{h'}{h}\sum_{i=1}^m |(\varphi_* {\sf e}_i)^\perp|^2 + g' \langle \bar\nabla\varrho, m{\bf H}\rangle + g''|\nabla (\varrho\circ\varphi)|^2\\
& = & h'\sum_{i=1}^m |(\varphi_* {\sf e}_i)^\perp|^2 + h \langle \bar\nabla\varrho, m{\bf H}\rangle + h'|\nabla (\varrho\circ\varphi)|^2\\
& = & h'\big(m-|\nabla (\varrho\circ\varphi)|^2\big)+ h \langle \bar\nabla\varrho, m{\bf H}\rangle + h'|\nabla (\varrho\circ\varphi)|^2,
\end{eqnarray*}
what proves (\ref{mainformula-reverse}).
\end{remark}

The weak maximum principle has the following consequence which is easily derived from inequality (\ref{mainformula-reverse}).

\begin{theorem}\label{thm4} Let $M$ be a geodesically and stochastically complete Riemannian manifold and
let $\varphi\colon M \to N$ be an isometric immersion of  $M$ into a complete Riemannian manifold $N$. Let $\Sigma_0$ be an oriented, complete hypersurface in $N$ and consider the signed distance function $\varrho = \dist(\cdot, \Sigma_0)$. Suppose that  $\varphi(M)$ is contained in some domain $\mathcal{U}\subset N$ where $\varrho$ is a regular function. Suppose that $\inf_\mathcal{U}h>0$ and that $\varphi(M)\subset \varrho^{-1}((-\infty, d))$ for some $d<+\infty$. If \emph{(\ref{comp-reverse1})} and \emph{(\ref{comp-reverse2})} in Theorem \ref{hessian}  hold then
\begin{equation}
\label{H-est}
\sup_M |{\bf H}|\ge\inf_{(d_*, d^*)}\frac{h'}{h}.
\end{equation}
\end{theorem}\begin{remark}If in Theorem \ref{thm4} we substitute $N$ by $N\times L$, $\Sigma_0$ by $\Sigma \times L$, $\varrho$ by $\widetilde \varrho$, where $L$ is a complete Riemannian manifold we have Theorem \ref{thm-product-intro}. \end{remark}

\begin{proof}\hspace{-1mm}: By hypothesis, $f$ is bounded above in $M$. The conclusion follows from (\ref{wmp}) applied to $f$ and from (\ref{mainformula-reverse}). \hfill $\square$
\end{proof}

\begin{remark}To obtain mean curvature  estimates of submanifolds via Proposition \ref{mainformula} one needs to choose judiciously a hypersurface $\Sigma_0$ satisfying all the requirements pointed above. In this next section we present several situations where we can choose   $\Sigma_0$ yielding good estimates.
\end{remark}

\begin{theorem}
\label{hypersurface}Let $\varphi\colon M \to N$ be a proper isometric immersion of a  complete Riemannian manifold $M$ into a complete Riemannian manifold $N$ with codimension $1$. Let $\Sigma_0$ be an oriented complete hypersurface in $N$ and  $\varrho = \dist(\cdot, \Sigma_0)$ the signed distance function. Suppose that  $\varphi(M)$ is contained in some domain $\mathcal{U}\subset N$ where $\varrho$ is a regular function.  Suppose that  the radial sectional and mean curvatures of $N$ and $\varphi(M)$ along the geodesics issuing perpendicularly from $\Sigma_0$ satisfies the inequalities \eqref{Krad2} and \eqref{HG2} and that
$\inf_\mathcal{U}h>0$.
 If $\varphi(M)\subset \varrho^{-1}((-\infty,d])$ for some $d>0$ and  $\varrho(\varphi(x_0))\ge 0$ for some $x_0\in M$ then
\begin{equation}
\sup_{M}|{\bf H}| \ge \inf_{\mathcal{U}}H_d.
\end{equation}
In particular, there are no stochastically complete hypersurfaces in $\varrho^{-1}((-\infty, d])$ satisfying $|{\bf H}|< \inf_\mathcal{U}H_d$.
\end{theorem}

\begin{proof}\hspace{-1mm}: We will apply a  localization procedure inspired by the techniques that A. Borb\'ely introduced in \cite{borbely-Bull-Aus}.  By assumption there exists $d\in (0,\infty)$ such that $\varrho\circ\varphi(x_0)\in [0,d]$ for some $x_0\in M$. Denote $f=g\circ\varrho\circ\varphi$ where
\begin{equation}
g(\varrho) = \int_{\varrho_0}^\varrho h(r)\, \textrm{d}r,
\end{equation} $\rho_0\geq d_{\ast}$.  Define $g_d\colon \mathcal{U}\subset N \to\mathbb{R}$ by
\begin{equation}
\tilde g_d(p)=\left\{\begin{array}{cl}
0, & \mbox{ if } \varrho(p)\in (-\infty, 0], \\
g(\varrho(p)), & \mbox{ if } \varrho(p)\in [0, d],\\
		   d, & \mbox{ if } \varrho(p)\in (d,+\infty]
              \end{array}\right.
\end{equation}
where $\varrho(p)={\rm dist}(p,\Sigma_0)$.
Hence the function $f_d = g_d\circ \varphi\colon M\to
\mathbb{R}$ is continuous, bounded above and smooth in
$(\varrho\circ\varphi)^{-1}(0,d)$. We may assume that $\sup_{M}\vert {\bf H}\vert < \infty$. In this situation, by
Theorem \ref{criterion2}, $M$ is stochastically complete and
there exists a sequence $x_k \in M$ such that $f_d(x_k)\to f_d^{\ast}$
and if $f_d$ is of class $C^{2}$ in a neighborhood of $x_k$ then
$\Delta  f_d(x_k)\leq 1/k$. However, $f(x_0)=g(\varrho\circ\varphi(x_0))\geq
f(x)$, for all $x\in (\varrho\circ\varphi)^{-1}(-\infty, 0)$. Therefore, we may
assume that  $x_{k}\in (\varrho\circ\varphi)^{-1}(0,d)$, where $f_d$ is smooth and coincides with $f$. Hence
applying  (\ref{Deltaf}), we get
$$
(n-1)|{\bf H}| +\frac{1}{kh}>  (n-1)H_d
$$
at the points of the sequence $\{x_k\}$ in $(\varrho\circ\varphi)^{-1}(0,d)$. Letting $k\to +\infty$ we have that $$\sup_{M}|{\bf H}|\geq \inf_{\mathcal{U}}  H_d.$$  This finishes the proof. \hfill $\square$\end{proof}

A sharp estimate also valid for higher codimension may be obtained if we explore some relation between the geometries of the ambient space $N$ and of the model space (see Remark \ref{geometries}). This is the case of the following result which jointly with the recent work by L. Mazet on the {\em vertical} half-space theorem in $E^3(-1,\tau)$, see \cite{mazet}, completes the picture in this setting of homogeneous three-dimensional geometries.

\begin{theorem}\label{submersion}
Let $\varphi\colon M \to N$ be a proper isometric immersion of a  complete Riemannian manifold $M$ into a complete Riemannian manifold $N$.
Let $\widetilde \Sigma_0$ be a  complete Riemannian manifold and suppose that there exists a Riemannian submersion $\pi:N\to \widetilde N$ with base space $\widetilde N=(d_*, d^*)\times_h \widetilde \Sigma_0$, where the warping function is a solution  $h:(d_*, d^*)\to \mathbb{R}$ of the ODE
\[
h''-Gh=0
\]
with $G$ satisfying conditions {\rm (\ref{G2})} and $h>0$ in $d_*<0< d^*$. Denote $\Sigma_0=\pi^{-1}(\widetilde\Sigma_0)$  and consider the signed distance function $\varrho = \dist(\cdot, \Sigma_0)$ in $N$.  We assume that the second fundamental form $A_0$ of $\Sigma_0$ satisfies $A_0\le \frac{h'(0)}{h(0)}g|_{T\Sigma_0}$.

Suppose that  $\varphi(M)$ is contained in some domain $\mathcal{U}\subset N$ where $\varrho$ is a regular function. Suppose that {\rm (\ref{Krad2})} and {\rm (\ref{HG2})} hold and that
$\inf_\mathcal{U}h>0$. If we have  $\varphi(M)\subset \varrho^{-1}((-\infty,d])$ for some $d>0$ and  $\varrho(\varphi(x_0))\ge 0$ for some $x_0\in M$ then
\begin{equation}
|{\bf H}| \ge \frac{m-1}{m}\inf_{\mathcal{U}}\frac{h'}{h} + \kappa,
\end{equation}
where $\kappa$ is the geodesic curvature of the vertical fibers of $\pi$.  In particular, there are no stochastically complete hypersurfaces in $\varrho^{-1}((-\infty, d])$ satisfying $|{\bf H}|<\disp \frac{m-1}{m}\inf_\mathcal{U}\frac{h'}{h}+\kappa$.\end{theorem}

\begin{proof}\hspace{-1mm}:
Since $\pi$ is a Riemannian submersion and $\Sigma_0 =\pi^{-1}(\widetilde\Sigma_0)$ the corresponding equidistant hypersurfaces in $N$ and $\widetilde N$ are also related by $\Sigma_d =\pi^{-1}(\widetilde\Sigma_d)$. In other terms,  $\varrho=\widetilde\varrho\circ\pi$, where $\widetilde\varrho=\dist_{\widetilde N}(\cdot, \widetilde\Sigma_0)$.

If we consider in $\Sigma_0$ a local orthonormal frame $\{{\sf v}_i\}_{i=1}^{n-1}$ such that the first $n-2$ vector fields are basic we have
\begin{eqnarray*}
(n-1)H_d &=& \sum_{i=1}^{n-2}\langle \bar\nabla_{{\sf v}_i}\bar\nabla \varrho, {\sf v}_i\rangle + \langle \bar\nabla_{{\sf v}_{n-1}}\bar\nabla \varrho, {\sf v}_{n-1}\rangle\\
&=& \sum_{i=1}^{n-2}\langle \nabla^{\widetilde N}_{\pi_*{\sf v}_i}\nabla^{\widetilde N} \widetilde\varrho, \pi_*{\sf v}_i\rangle+ \langle \bar\nabla_{{\sf v}_{n-1}}\bar\nabla \varrho, {\sf v}_{n-1}\rangle\\
& = & (n-2)\widetilde H_d  + \langle \bar\nabla_{{\sf v}_{n-1}}\bar\nabla \varrho, {\sf v}_{n-1}\rangle,
\end{eqnarray*}
where $\nabla^{\widetilde N}$ is the Riemannnian covariant derivative in $\widetilde N$ and $\widetilde H_d$ is the mean curvature of $\widetilde \Sigma_d$ with respect to $-\nabla^{\widetilde N}\widetilde \varrho$.  However by construction we have
\[
\widetilde H_d = \frac{h'(d)}{h(d)}.
\]
Hence defining a function $f_d$ as in the proof of Theorem \ref{hypersurface} and using (\ref{Deltaf}) we get
\begin{equation}
m|{\bf H}| +\frac{1}{kh} > (m-1)\frac{h'}{h}+\kappa,
\end{equation}
where
\[
\kappa = \inf_{{\sf v}}\langle \bar\nabla_{{\sf v}}\bar\nabla\varrho, {{\sf v}}\rangle.
\]
the infimum being taken over the vertical tangent vectors ${\sf v}$ in $T\mathcal{U}$.  This finishes the proof. \hfill $\square$  \end{proof}

\begin{theorem}\label{submersion-hyp} Let $N$ be a {\rm $(n+\ell)$}-dimensional complete Riemannian manifold which is a total space of a Riemannian submersion $\pi\colon N^{n+\ell}\to \mathbb{H}^n$. Suppose that the sectional curvatures of $N$ satisfy $K_N\ge -1$ everywhere.
Let $\varphi\colon M \to N$ be a proper isometric immersion of a  complete Riemannian manifold $M$ into  $N$ with bounded mean curvature $|{\bf H}|$.
Let $\widetilde\Sigma_0$ be a horosphere in  $\mathbb{H}^n$ and denote $\Sigma_0=\pi^{-1}(\widetilde \Sigma_0)$.  Suppose that the second fundamental form  $A_0$ of  $\Sigma_0$ satisfies $A_0 \le g|_{T\Sigma_0}$.
Suppose  that $\varphi(M)$ is contained in one of the connected components of $N\backslash \Sigma$ which we denote by $\mathcal{C}$.
Then
\begin{equation}
|{\bf H}| \ge \frac{m-\ell}{m}+ \frac{\ell}{m} \inf_{\mathcal{C}}\kappa,
\end{equation}
where $\kappa$ is the mean curvature of the vertical fibers of $\pi$ with respect to $\bar\nabla\varrho$.  In particular, there are no stochastically complete hypersurfaces in $\mathcal{C}$ satisfying $|{\bf H}|< \disp\frac{m-\ell}{m}+\frac{\ell}{m}\kappa$.\end{theorem}

\begin{proof}\hspace{-1mm}: In terms of the notation fixed earlier in Section \ref{sec:comparison}, we have now $G=1$.  The base space $\mathbb{H}^n$ may be described in terms of  Fermi coordinates over $\widetilde \Sigma_0$ as the warped space $\mathbb{R}\times_h \widetilde \Sigma_0$ where $h(d)= e^{d}$.  The assumption $\varphi(M)\subset \mathcal{C}$ may be stated as $M \subset (\varrho\circ\varphi)^{-1}((0,+\infty))$. Choose $d>0$ in such a way that there exists $x_0\in M$ with $\varphi(x_0)\in \varrho^{-1}((0,d))$.  Hence truncating $f=g\circ\varphi$ at the level $d$ as in Theorem \ref{hypersurface} and using (\ref{Deltaf})
\begin{eqnarray*}
m|{\bf H}|+\frac{1}{kh} &\ge & (n+\ell-1)H_d-(n+\ell-1)\frac{h'}{h} + m\frac{h'}{h}\\
& \ge & (n-1) \frac{h'}{h} +\sum_{i=1}^\ell\langle \bar\nabla_{{\sf v}_i} \bar\nabla\varrho, {\sf v}_i \rangle  -(n+\ell-1)\frac{h'}{h} + m\frac{h'}{h}\\
& \ge & (m-\ell)\frac{h'}{h} +\ell\inf_{p\in \mathcal{C}} \kappa_p,
\end{eqnarray*}
where $\{{\sf v}_i\}_{i=1}^\ell$ is a local vertical orthonormal frame and $\kappa_p$ is the mean curvature of the vertical fiber $\pi^{-1}(p)$ with respect to the horizontal vector field $\bar\nabla\varrho$. In other terms, $\ell\kappa$  is the trace of the O'Neill's tensor $T_{\bar\nabla\varrho}$.
This concludes the proof. \hfill $\square$
\end{proof}

In the case that $N$ is a Riemannian product  $N\times L$ a variant of the calculations preceding Theorem \ref{thm4} involves only the sectional curvature of the basis $N$. Indeed we
have Theorem \ref{thm-product-intro} stated in the introduction. Here we present its proof for the sake of completeness.

\begin{proof}\hspace{-1mm}:
The proof follows the same guidelines of the proof of Proposition \ref{mainformula}.  We only need to adjust the calculations for using the Hessian comparison theorem on $N$ instead of $N\times L$. We have
\begin{eqnarray*}
\Delta  f & = & g'\sum_{i=1}^m \langle \bar\nabla_{\varphi_* {\sf e}_i} \bar\nabla \varrho, \varphi_*{\sf e}_i\rangle + g' \langle \bar\nabla\varrho, m{\bf H}\rangle + g''|\nabla (\varrho\circ\varphi)|^2\\
& = & g'\sum_{i=1}^m \langle \nabla^N_{\pi_*\varphi_* {\sf e}_i} \nabla^N \varrho_N, \pi_*\varphi_*{\sf e}_i\rangle + g' \langle \bar\nabla\varrho, m{\bf H}\rangle + g''|\nabla (\varrho\circ\varphi)|^2\\
& \ge & g'\frac{h'}{h}\sum_{i=1}^m |\pi_*\varphi_* {\sf e}_i^\perp |^2+ g' \langle \bar\nabla\varrho, m{\bf H}\rangle + g''|\nabla (\varrho\circ\varphi)|^2\\
&=& h'\sum_{i=1}^m |\pi_*\varphi_* {\sf e}_i^\perp |^2+ h \langle \bar\nabla\varrho, m{\bf H}\rangle + h'|\nabla (\varrho\circ\varphi)|^2
\end{eqnarray*}
In the inequality we used the Hessian comparison theorem for the distance function in $N$.
Using the fact that $\sum_{i=1}^{m}|(\pi_{\ast}\varphi_* {\sf e}_i)^{\perp}|^{2} + |\nabla\varrho\circ\varphi|^2 = m -
\sum_{i=1}^{m}|\pi^{L}_{\ast}\varphi_* {\sf e}_i| \ge m-\ell$, that $f$ is bounded above and that $M$ is stochastically complete, we
have at a maximizing sequence $x_k$ that 
$$
\begin{aligned}
\frac{1}{kh} + m|{\bf H}|  & \geq \frac{h'}{h}\big(m-\ell-|\nabla\varrho\circ\varphi|^2\big) + \frac{h'}{h}|\nabla (\varrho\circ\varphi)|^2
\\
& = \frac{h'}{h}(m-\ell).
\end{aligned}
$$
Finally, letting $k\to\infty$ we have
$$
\sup_{M} \vert {\bf H}\vert \geq \frac{(m-\ell)}{m} \inf \frac{h'}{h}.
$$
This concludes the proof. \hfill $\square$\end{proof}

\section{Half-space theorems in $\mathbb{H}^{n}\times \mathbb{R}^{\ell}$}\label{half-space-section}
The celebrated Strong Half-Space Theorem proved by  Hoffmann-Meeks
\cite{hoffmann-meeks} states that two complete, minimally and
properly immersed  surfaces of $\mathbb{R}^{3}$ must intersect,
unless they are parallel planes. The word {\em strong} here is to contrast with  the (weak) half-space theorem,
proved by W. Meeks in \cite{meeks}, which states that a  complete surface $M$ can not be immersed properly and minimally
  into a half-space $\{ x_i>0 \}\subset \mathbb{R}^{3}$   unless $M$ is a plane parallel to $\{x_i=0\}$.

These  half-space theorems were also proved in  the class of complete
surfaces with bounded sectional curvature, minimally immersed into
$\mathbb{R}^{3}$.  First, F. Xavier  \cite{xavier}  proved the weak half-space theorem, this is, a
complete  surface $M$ with bounded  curvature can not be  minimally immersed  into a half-space $\{x_i>0\}$ unless $M$ is
a plane parallel to the plane $\{x_i=0\}$. The Strong Half-Space Theorem in this setting (bounded curvature) was settled independently by
Bessa-Jorge-Oliveira \cite{bessa-jorge-oliveira} and
by  H. Rosenberg \cite{rosenberg-hst}. They  proved that two
complete   minimal surfaces with bounded  curvature intersect,
unless they are parallel planes.  Bessa-Jorge-Oliveira
also proved  the {\em
Mixed Half-Space Theorem}, see \cite{bessa-jorge-oliveira}, that states that  a complete  proper minimal
surface and a complete minimal surface with bounded curvature must
intersect unless they are parallel planes.

Recently, the theory of minimal surfaces was successfully extended to product spaces
$N\times \mathbb{R}$, where $N$ is a complete Riemannian surface. This extension started to
be developed after the seminal paper \cite{rosenberg-HXR} by H. Rosenberg. In some extent the study of minimal/cmc surfaces in $N\times\mathbb{R}$
is guided by the classical theory in $\mathbb{R}^{3}$. However, it is very sensitive to the
 geometry of  $N$,  yielding very different results from their classical counterparts, see, for instance,
\cite{rosenberg-HXR}, \cite{rosenberg-book} and references
therein.  In this spirit,
Hauswirth, Rosenberg and Spruck \cite{hauswirth-rosenberg-spruck}
proved a version of the Meeks' Half-Space Theorem in
$\mathbb{H}^{2}\times \mathbb{R}$.

\begin{theorem}[Hauswirth-Rosenberg-Spruck]
Let $\Sigma$ be a properly embedded constant mean curvature
$H=\frac{1}{2}$ surface in $\mathbb{H}^{2}\times \mathbb{R}$.
Suppose $\Sigma$ is contained in a horocylinder $C=\mathbb{B}\times
\mathbb{R}$ and is asymptotic to $C$, where  $\mathbb{B}$ is a
horoball. If the mean curvature vector of $\Sigma$ has the same
direction as that of $C$ at points of $\Sigma$ converging to $C$
then $\Sigma$ is equal to $C$ or is a subset of $C$ if $\partial
\Sigma\neq \emptyset$.\label{HRS-Thm}
\end{theorem}

In $\mathbb{H}^2\times\mathbb{R}$, the horocylinders and the  surfaces with constant mean curvature $H=\frac{1}{2}$  play, respectively, the role of the planes and the minimal surfaces  in the half-space theorems in $\mathbb{R}^{3}$. In this paper we prove an extension Theorem \ref{HRS-Thm}, in many aspects, in the following  half-space theorem in $\mathbb{H}^{2}\times \mathbb{R}$.

\begin{theorem}\label{thmMain1}
Let $\Sigma\hookrightarrow \mathbb{H}^{2}\times \mathbb{R}$ be a properly immersed surface in
$\mathbb{H}^{2}\times \mathbb{R}$.  If $\Sigma$ is contained in a
horocylinder $C=\overline{\mathbb{B}}\times \mathbb{R}$ with an
interior point, $\mathbb{B}\subset \mathbb{H}^{2}$ a horoball,
then $\sup \vert {\bf H} \vert \geq \frac{1}{2}$. In particular, if
$\Sigma$ is  contained in a horocylinder
$C=\overline{\mathbb{B}}\times \mathbb{R}$ and  $\vert {\bf H} \vert
<\frac{1}{2}$ then $\Sigma\subset \partial C=\partial
\mathbb{B}\times \mathbb{R}$.
\end{theorem}

 Clearly, Theorem \ref{thmMain1} is a direct consequence  of the following mean curvature estimate in  the class of stochastically complete $m$-submanifolds  of $\mathbb{H}^{n}\times
\mathbb{R}^{\ell}$, $m\geq \ell +1$, since we may suppose that $\sup \vert{\bf  H}\vert < \infty$ and in this case, by  Theorem  \ref{criterion2}, $M$ is stochastically complete.

\begin{theorem}\label{thmMain2}
Let $\varphi\colon M^m \to \mathbb{H}^{n}\times
\mathbb{R}^{\ell}$, $m\geq \ell +1$, be an isometric immersion of
a stochastically complete $m$-dimensional Riemannian manifold into
 $\mathbb{H}^{n}\times \mathbb{R}^{\ell}$. Let
$\mathbb{B}\subset \mathbb{H}^{n}$ be a horoball and let
$C=\overline{\mathbb{B}}\times \mathbb{R}^{\ell}$ be a generalized
horocylinder. If $\varphi (M)\subset C$ has an interior point,
then the mean curvature of $\varphi(M)$ satisfies \[\sup_{M}\vert {\bf
H}\vert \geq \frac{m-\ell}{m}\cdot\]
\end{theorem}
\begin{remark}We should remark that Theorem \ref{thmMain2} can be also obtained using the geometry of the Busemann functions of $\mathbb{H}^{n}$, \cite{pigola}. In fact, in this point of view, Theorem \ref{thmMain2}  can be extended to more general Hadamard manifolds, \cite{blsp}.
\end{remark}

\begin{proof}\hspace{-1mm}: By assumption, we have  $\varphi (M) \subset
C=\overline{\mathbb{B}}\times \mathbb{R^{\ell}}$,  $\mathbb{B}\subset \mathbb{H}^{n}$ a horoball with ideal point $p_{\infty}$ and    with an interior
point $\varphi(x_0)=p_0\in \varphi (M)\cap (\mathbb{B}\times
\mathbb{R^{\ell}})$.  Assume without loss of generality that
$p_0=(z_0, 0)\in \mathbb{B}\times \{0\}$.   Let $\mathbb{B}_0$ be a horoball in $\mathbb{H}^n$ with ideal point $q_{\infty}$, distinct from $p_{\infty}$.  For instance, consider these points as antipodal in the disc model of $\mathbb{H}^n$. Fix $\Sigma_0=\partial \mathbb{B}_0\times \mathbb{R}^l$, the horocylinder over  $\partial\mathbb{B}_0$. Consider an equidistant horocylinder  $\Sigma_{d'}$ which also intersects
$\mathbb{B}\times \mathbb{R}^l$.  Settled these choices, the conclusion comes from Theorem \ref{thm-product-intro}  if we recall that $h(d)=\sinh d$. \hfill $\square$
\end{proof}

An immediate corollary of Theorem \ref{thmMain2} is the following result that can be regarded as a Mixed Half-Space Theorem in $\mathbb{H}^{n}\times \mathbb{R}^{\ell}$. \begin{theorem}\label{thmMain3}
Let $\Sigma$ be an immersed submanifold in $\mathbb{H}^{n}\times
\mathbb{R^{\ell}}$, $m\geq \ell+1$, with scalar curvature
satisfying
\[
s_{\Sigma}(x)\geq - c^{2}\, \rho_{\Sigma}^{2}(x)\log
(\rho_{\Sigma}(x) +1),\,\,\, \mbox{for }\rho_{\Sigma}\gg 1 \mbox{
and for some } c\in \mathbb{R}.
\]
Let $C$ be either a generalized horocylinder
$C=\overline{\mathbb{B}}\times \mathbb{R^{\ell}}.$ If $\Sigma \subset C$ has an interior point then $$\sup \vert {\bf H}
\vert \geq \displaystyle\frac{m-\ell}{m}.$$Here $p\in \Sigma$ and $\rho_{\Sigma}(x)={\rm
dist}_{\Sigma}(p,x)$.
\end{theorem}

\begin{proof}\hspace{-1mm}: It
suffices to see that a Riemannian manifold $\Sigma$ with scalar
curvature \[s_{\Sigma}(x)\geq - c^{2}\cdot
\rho_{\Sigma}^{2}(x)\log (\rho_{\Sigma}(x)
+1),\quad\rho_{\Sigma}\gg 1,\,\,\, c\in \mathbb{R}\] immersed into a
complete Riemannian manifold with sectional curvature bounded
above is stochastically complete. See the proof of
\cite[Thm. 1.15]{prs-memoirs}. \hfill $\square$
\end{proof}

\section{Horizontal half-space theorems}

I. Salavessa in \cite{Salavessa-thesis}, \cite{salavessa}, constructed for each  $c\in (0, n-1)$ a smooth radial function $u\colon \mathbb{H}^{n}\to \mathbb{R}$ whose graph $\Gamma_n(\frac{c}{n})=\{(x, u(x))\in \mathbb{H}^{n}\times \mathbb{R}\}$ has constant mean curvature $ c/n$. R. S. Earp and E. Toubiana \cite{earp-toubiana} found explicit formulas for rotational surfaces of $\mathbb{H}^{2}\times \mathbb{R}$ with constant mean curvature $\vert H \vert = \frac{1}{2}$. These surfaces are indexed by $\mathbb{R}_{+}$, meaning that for each $\alpha\in \mathbb{R}_{+}$ there exists a surface $\mathcal{H}_{\alpha}\subset \mathbb{H}^{2}\times \mathbb{R}$ with constant mean curvature $\vert H \vert = \frac{1}{2}$. For $\alpha <1$ the surface has two vertical ends. For $\alpha >1$, $\mathcal{H}_{\alpha}$ is not embedded and $\mathcal{H}_{1}$ has only one end and it is a graph over $\mathbb{H}^{2}$. It turns out that $\mathcal{H}_{1}=\Gamma_{2}(\frac{1}{2})$, the Salavessa's surface.
 Earp and  Nelli \cite{earp-nelli}
obtained  a version of Meeks Half-Space Theorem for $\mathbb{H}^{2}\times \mathbb{R}$ in a different perspective from Hauswirth-Rosenberg-Spruck. They proved the
following result.
\begin{theorem}[Earp-Nelli]\label{earp-nelli}
Let $\Sigma\subset \mathbb{H}^{2}\times \mathbb{R}$ be a complete surface with constant
mean curvature $\vert H\vert=\frac{1}{2}$, $\Sigma$ different from the Salavessa Graph in dimension $2$,  $\Gamma_{2}(\frac{1}{2})$. Then $\Sigma $
can not be properly immersed in the mean convex side of $\Gamma_{2}(\frac{1}{2})$.
\end{theorem}

We now give a brief description of those rotationally invariant graphs with constant mean curvature.
Let $(r,\vartheta, t)\in  [0,\infty)\times\mathbb{S}^{n-1}\times \mathbb{R}$ be cylindrical coordinates in $\mathbb{H}^n\times \mathbb{R}$ defined with respect to the axis $\ell=\{o\}\times \mathbb{R}$ for some $o\in \mathbb{H}^n$. A rotationally invariant graph is described in terms of these coordinates by a real function $t= u(r)$. It may be easily verified that such a graph has constant mean curvature $H$ (with respect to the upward orientation) if and only if
\begin{equation}
\label{flux-s}
I =  \sinh^{n-1}r \frac{u'}{\sqrt{1+u'^2}} - nH \int_0^{r} \sinh^{n-1}(\tau)\, d\tau
\end{equation}
for some constant $I$. If the graph reaches the axis $\ell$ orthogonally then $I=0$.  This is the case of a family of hemispheres with constant mean curvature varying in the entire range $(\frac{n-1}{n},+\infty)$. Indeed we have that
$u'(r)\to\infty$ as $r\to r_0$ for some $r_0\in (0,\infty]$ if and only if
\begin{eqnarray*}
1= nH \frac{\int_0^{r_0} \sinh^{n-1}(\tau)\, d\tau}{\sinh^{n-1}(r_0)}.
\end{eqnarray*}
Notice that there exists such a critical value $r=r_0$ if and only if $H\in (\frac{n-1}{n},+\infty)$, an interval that parameterizes the family of graphs defined  by the solutions of the ODE
\begin{equation}
u'^2 = \disp\frac{-nH\Big(\disp\frac{\int_0^{r} \sinh^{n-1}(\tau)\, d\tau}{\sinh^{n-1}(r)}\Big)^2}{1-n^2H^2 \Big(\disp\frac{\int_0^{r} \sinh^{n-1}(\tau)\, d\tau}{\sinh^{n-1}(r)}\Big)^2}.
\end{equation}
These graphs reach the axis $\ell$ orthogonally.  Reflecting this rotationally symmetric graph through the vertical slice at its maximum height, it yields a constant mean curvature hypersurface diffeomorphic to a sphere.

At the limit value $H=\frac{n-1}{n}$ we have $|u'(r)|<\infty$ in $[0,+\infty)$. Hence the graph generates a disc  with constant mean curvature $H=\frac{n-1}{n}$ and a unique vertical end. This disc separates $\mathbb{H}^n\times \mathbb{R}$ into two components and coincides with $\Gamma_n(\frac{n-1}{n})$ for $H=\frac{n-1}{n}$. The component in the mean convex side $\mathcal{H}$ of $\Gamma_n(\frac{n-1}{n})$ is foliated by the spheres described above with constant mean curvature $H\in (\frac{n-1}{n},+\infty)$. These hypersurfaces have a common tangency point in $o\in \mathbb{H}^n\times \{0\}$.

\begin{theorem}\label{thmMain4bis}
Let $\Sigma$ be a  hypersurface properly immersed into the mean convex side of $\Gamma_n(\frac{n-1}{n})$.
Then $\sup_{\Sigma}\vert H
\vert > \frac{n-1}{n}$.
\end{theorem}

\begin{remark}The Theorem \ref{thmMain4bis} holds under the hypothesis ``\emph{stochastically complete}'' instead of  ``\emph{properly immersed}''.
\end{remark}

\begin{proof}
The proof of Theorem \ref{thmMain4bis} is similar to the proof of
Theorem \ref{submersion-hyp}. We will reproduce  it here for sake of
completeness.

By assumption, there exists $x_0\in M$ such that $\varphi(x_0)\in \mathcal{H}$. There exists a geodesic sphere $\widetilde\Sigma_0$ in $\mathbb{H}^n$ with radius $d_0$ such that the cylinder $\Sigma_0 = \widetilde\Sigma_0\times \mathbb{R}$ contains points of $\varphi(M)$ in its mean convex side.  Then fix $d<d_0$ such that $\varphi(M)\cap \varrho^{-1}((0,d))\neq \emptyset$.  In terms of the notation fixed earlier in this text, define the truncation of $f$ as in Theorem \ref{hypersurface}. Using (\ref{Deltaf}) we obtain
\begin{equation*}
m|{\bf H}|+\frac{1}{kh} \geq (n-2)\frac{h'}{h}-(n-1)\frac{h'}{h} + m\frac{h'}{h} = (m-1)\frac{h'}{h}\geq (m-1)\coth d_0 >(m-1),
\end{equation*}
what concludes the proof. \hfill $\square$\end{proof}

\section{Wedgedly bounded submanifolds}\label{sec:wedge}

Let $N^n$ be a Riemannian manifold with a pole $p$. The non-degenerate cone in $N$ with vertex  at $p$ determined by  ${\sf v}\in T_p N$, $|{\sf v}|=1$, and $a\in (0,1)$ is defined by $$\mathcal{C}(p, {\sf v},a)=\left\{\exp_{p}(t{\sf w})\in N:\, t\geq 0,\, {\sf w}\in T_{p}N,\,\vert {\sf w}\vert=1, \,\langle {\sf v},\,{\sf w}\rangle \geq a\right \}.$$

\begin{definition} Let $L^\ell$ be a complete Riemannian manifold.   The wedge determined by ${\sf v}$ and $a$ is the subset $\mathcal{W}(p,{\sf v},a)\subset N\times L$ of the Riemannian product $N\times L$
 defined by  $$\mathcal{W}(p,{\sf v},a)=\mathcal{C}(p,{\sf v},a)\times L.$$ \end{definition}

A. Borb\'{e}ly in \cite{borbely-Bull-Aus} proved that complete minimal immersion $\varphi \colon M \to \mathbb{R}^{3}$ satisfying the Omori-Yau maximum principle can not be contained in a wedge of $\mathbb{R}^{3}$. His ideas could be easily adapted to obtain mean curvature estimates for
{\em wedgedly} bounded, stochastically complete submanifolds of
$B\times F$.

\begin{theorem}\label{thmMain5}
Let  $\varphi \colon M \to N\times L$
an isometric immersion of a stochastically complete $m$-dimensional  manifold
$M$, where $m\geq \ell+1$. Suppose that $\varphi(M)$ is a submanifold of a wedge $\mathcal{W}(p,{\sf v},a)$, for some ${\sf v}\in T_p N$, $|{\sf v}|=1$, and $a\in (0,1)$.  Suppose that the sectional curvatures of $N$ satisfy
\begin{equation}
K_{N} \le -c^2\,\, \mbox{ in } \,\, \mathcal{C}(p,{\sf v},a),
\end{equation}
 for some constant $c\ge 0$.
Then $$\sup_{M} \vert {\bf  H}\vert  > 0 \,\,\,
\mathrm{if}\,\,\, c=0\quad \mathrm{and}\quad \sup_{M} \vert {\bf H}\vert
\geq  \displaystyle\frac{(m-\ell)c}{m} \,\,\,
\mathrm{if}\,\,\, c>0.$$
 \end{theorem}
\begin{remark}Sara Boscatti in her master thesis \cite{sara} proved Theorem \ref{thmMain5} for wedgedly bounded submanifolds of the Euclidean space.
\end{remark}

\begin{proof}\hspace{-1mm}:
Let $\gamma (t) =\exp_{p}^{N}(t{\sf v})$ be the geodesic of $N$ satisfying $\gamma(0)=p$, $\gamma'(0)={\sf v}$  called  the axis of $\mathcal{C}_{N}(p, {\sf v},a)$.

Consider $y_0\in L$ such that  there exists a point  $\bar{q}=\varphi (q)\in \mathcal{C}_{N}(p, {\sf v},a)\times\{y_0\}$.
Let $B^{N}_{\gamma(t_0)} (r_{\bar{q}})$ be the geodesic ball of $N$ with center at $\gamma(t_0)$ and radius $r_{\bar{q}}={\rm dist}_{N}(\bar{q}, (\gamma(t_0), y_0))$. For  $t_0\gg 1$ the set $\mathcal{C}_{N}(p,v,a)\setminus B^{N}_{\gamma(t_0)}(r_{\bar{q}})$  has two connected components. Taking a radius $r>0$ slightly smaller than $r_{\bar{q}}$  we have that the set $\mathcal{W}(p,{\sf v},a)\setminus B^{N}_{\gamma(t_0)}(r)\times L$ also has two connected components $\mathcal{W}_1$ and $\mathcal{W}_2$, one of them, say $\mathcal{W}_{1}$, containing  $\bar{q}$.

Hence we fix $\Sigma_0 =\partial B^{N}_{\gamma(t_0)}(r)\times L$ and define $\widetilde\varrho = \textrm{dist}_{N\times L}(\cdot, \Sigma_0)$. Note that
$\widetilde\varrho(q) = \varrho (\pi(q))$ where  $\pi:N\times L \to L$ is the standard projection and $\varrho = \textrm{dist}_N(\cdot, \partial B^{N}_{\gamma(t_0)}(r))$. We define
$$
g_{r}(q)=\left\{ \begin{array}{lll} g(\varrho(q)) & \mathrm{if}& q  \in  \mathcal{W}_1\cup  B^{N}_{\gamma(t_0)}(r)\times L\\
g(r) & \mathrm{if}& q  \in \mathcal{W}_2
\end{array}\right.
$$
where
$$
g(t)=\left\{ \begin{array}{lll}
t^{2}/2 & \mathrm{if} &  c=0,\\
\cosh(ct)/c & \mathrm{if} & c>0.
\end{array}\right.
$$

It follows that the function $f:M\to \mathbb{R}$ defined by $f=g_{r}\circ \varrho\circ\varphi$ is continuous, bounded above and smooth in $\mathcal{W}_1$. In order to compute the Laplacian of $f|_{\varphi^{-1}(\mathcal{W}_1)}$ we proceed as in the proof of Theorem \ref{thm-product-intro}.
Computing $\Delta f(x_k)$ along a maximizing sequence $\{x_k\}\in \varphi^{-1}(\mathcal{W}_1)$ and  letting $k\to\infty$ we obtain
$$
\sup_{M} \vert {\bf H}\vert \geq  \frac{(m-\ell)}{m}\, \inf_{[0, t_0] }\frac{h_N'}{h_N},
$$
where $h_N (t) = g'(t)$. This concludes the proof. \hfill $\square$

\end{proof}

\begin{remark} These estimates are similar to those ones found by Bessa, Lima and Pessoa in \cite{bessa-lima-pessoa} for $\phi$-bounded submanifolds, however, the class of {\em wedgedly} bounded
submanifolds are more general than the class of $\phi$-bounded
considered by them.
 \end{remark}

\begin{acknowledgements}
The authors want to express their gratitude to the  Universidade Federal de Alagoas-UFAL for their support during the event {\em  III Workshop de Geometria Diferencial} where this work started.
\end{acknowledgements}

\end{document}